\documentclass[[amscd,amssymb,verbatim,12pt]{amsart}
\usepackage{amssymb,latexsym, amsmath, amscd, array, graphicx, relsize}

\swapnumbers \numberwithin{equation}{section}

\theoremstyle{break}

\newtheorem{thm}{Theorem}[section]

\newtheorem{lemma}[thm]{Lemma}

\newtheorem{prop}[thm]{Proposition}
\newtheorem{cor}[thm]{Corollary}

\newtheoremstyle{break}% name
   {9pt}%      Space above, empty = `usual value'
   {9pt}%      Space below
   {\rmfamily}% Body font
   {}%         Indent amount (empty = no indent, \parindent = para indent)
   {\bfseries}% Thm head font
   {.}%        Punctuation after thm head
   {\newline}% Space after thm head: \newline = linebreak
{}% Thm head spec

\theoremstyle{definition}
\newtheorem{defn}[thm]{Definition}

 \newcommand{\Wi}{\widetilde}

\DeclareMathOperator{\Int}{{\rm Int}}

\def\Int{\protect\operatorname{Int}}

\def\C{{\mathbb C}}
\def\Z{{\mathbb Z}}

\def\1{\hbox{\rm\rlap {1}\hskip.03in{\rom I}}}
\def\Bbbone{{\rm1\mathchoice{\kern-0.25em}{\kern-0.25em}
{\kern-0.2em}{\kern-0.2em}I}}

\long\def\forget#1\forgotten{} %

\newcommand\ver[1]{\marginpar{\tiny Changed in Ver \VER}}

\newcommand{\mc}{ \text {mc}}

\newcounter{notecounter}

\date{\today}

\begin{document}

\title[On Macroscopic dimension of non-spin  4-manifolds]{On Macroscopic dimension of non-spin  4-manifolds}

\author[M.~Daher]{Michelle Daher}

\author[A.~Dranishnikov]{Alexander  Dranishnikov}

\address{Michelle Daher, Department of Mathematics, University
of Florida, 358 Little Hall, Gainesville, FL 32611-8105, USA}
\email{mmdaher@ufl.edu}

\address{Alexander N. Dranishnikov, Department of Mathematics, University
of Florida, 358 Little Hall, Gainesville, FL 32611-8105, USA}
\email{dranish@ufl.edu}



\begin{abstract}
We prove that for 4-manifolds $M$ with residually finite fundamental group and non-spin universal covering $\Wi M$, the inequality
$\dim_{mc}\Wi M\le 3$ implies the inequality $\dim_{mc}\Wi M\le 2$. This allows us to complete the proof of Gromov's Conjecture for 4-manifolds
with abelian fundamental group.
\end{abstract}

%%% ----------------------------------------------------------------------


%%% ----------------------------------------------------------------------
\maketitle
%%% ----------------------------------------------------------------------

\section {Introduction}
 M. Gromov~\cite{G2} introduced the concept of macroscopic dimension to describe some large scale phenomenon of universal covering of  manifolds with positive scalar curvature.
He discovered some large scale dimensional deficiency of the universal covering of such manifolds which he formulated in the following:

{\bf Gromov's Conjecture.} {\it The macroscopic dimension of the universal covering $\Wi M$ of  a closed   $n$-manifold $M$ with positive scalar curvature satisfies the inequality $\dim_{mc}\Wi M\leq n-2$ for a metric on $\Wi M$ lifted from $M$}.

\begin{defn}\cite{G2}
 A metric space $X$ has macroscopic dimension $\dim_{\mc} X \leq k$ if
there is a uniformly cobounded proper continuous map $f:X\to K$ to a $k$-dimensional simplicial complex.
Thus, the macroscopic dimension of $X$ is $m$, $\dim_{mc}X=m$, where $m$ is minimal among $k$ with $\dim_{\mc} X \leq k$.
\end{defn}
\smallskip
We recall that a map of a metric space $f:X\to Y$ is uniformly cobounded if there is a uniform upper bound on the diameter of preimages $f^{-1}(y)$, $y\in Y$.

\

Since $\dim_{mc}X=0$ for every compact metric space, Gromov's Conjecture holds trivially for simply connected manifolds. In this paper we complete the proof of Gromov's Conjecture for 4-manifolds with abelian fundamental group. Now Gromov's Conjecture for manifolds with an abelian fundamental group is established in all dimensions modulo the status of the quite technical preprint~\cite{SY2}.
First, the conjecture was settled in all dimensions for
almost spin  manifolds with an abelian fundamental group~\cite{Dr1}. For a closed manifold $M$ of dimension $n \geq5$ with an abelian fundamental group and a non-spin universal cover, the problem was reduced  in~\cite{BD} (Theorem 4.7 and Corollary 4.11) to a manifold of type $T^{4n}\#\C P^{2n}$. Theorem 1.2 from a recent preprint of Schoen and Yau~\cite{SY2} states that
$T^{4n}\#\C P^{2n}$ cannot carry a metric of positive scalar curvature. This covers the remaining case for $n\ge 5$.

Gromov's conjecture can be split in two steps: The first step is to show the inequality  $\dim_{mc}\Wi M\leq n-1$ in the case of a positive scalar curvature metric on $M$ and
the next step is to prove the inequality $\dim_{mc}\Wi M\leq n-2$. The first step of this conjecture can be traced back to~\cite{G1}, and it seems
out of reach in view of the fact that it implies the Gromov-Lawson Conjecture:
{\em A closed aspherical manifold cannot support a metric of positive scalar curvature}. This conjecture is a relative of the famous long standing Novikov Higher Signature conjecture~\cite{R1}.
The second step seemed to be manageable.
In fact, Gromov proposed the following~\cite{G2}:

{\bf Conjecture.} {\em The inequality $dim_{mc}\Wi M\le n-1$ implies the inequality $dim_{mc}\Wi M\le n-2$ for all manifolds.}

It was known that this conjecture  holds true for 3-manifolds~\cite{B2}~\cite{GL}.
In~\cite{B1} D. Bolotov constructed a 4-dimensional counterexample.
In~\cite{BD} we proved that this conjecture holds true for $n$-manifolds $M$, $n\ge 5$, whenever the universal cover $\Wi M$ is not spin. In this paper we prove this
conjecture for 4-manifolds with residually finite fundamental group and non-spin universal cover.
\begin{thm}
Let $M$ be a closed 4-manifold with residually finite fundamental group whose universal cover $\Wi M$ is not spin and $\dim_{mc}\Wi M\le 3$. Then $\dim_{mc}\Wi M\le 2$.
\end{thm}
The totally non-spin condition in this theorem is crucial in view of Bolotov's example~\cite{B1} of a spin 4-manifold $M$ with $\dim_{mc}\Wi M=3$ and residually finite fundamental group
$\pi_1(M)=\Z\ast (\Z\times\Z)$.

We note that for $n=\dim M>4$ such theorem was proven in~\cite{BD} without restriction on the fundamental group with the use
of the following theorem of Wall~\cite{W}, which holds only for $n\ge 5$.
\begin{thm}\label{Wa}
Let $W$ be a bordism between compact $n$-manifolds, $n\ge 5$, $M$ and $N$ which is stationary on the boundary, $\partial M=\partial N$.
Suppose that the inclusion $M\to W$ is a $k$-equivalence. Then $W$ admits a handle decomposition with no handles of index $\le k+1$.
\end{thm}
We recall that a map $f:X\to Y$ is called a $k$-equivalence if $f_*:\pi_i(X)\to\pi_i(Y)$ is an isomorphism for $i\le k$ and an epimorphism for $i=k+1$.

\

\

\section{Preliminaries}

Let $\pi=\pi_1(K)$ be the fundamental group of a CW complex $K$. By $u^K:K\to B\pi=K(\pi,1)$ we denote a map that classifies the universal covering $\Wi K$
of $K$. We refer to $u^K$ as a {\em classifying map} for $K$. We note that a map $f:K\to B\pi$ is a classifying map if and only if it induces an isomorphism on the fundamental groups.

\subsection{Inessential manifolds and macroscopic dimension} We recall the following definition of Gromov~\cite{G3}:

\begin{defn}
An $n$-manifold $M$  with fundamental group $\pi$ is called {\it inessential}  if its classifying map $u^M:M\to B\pi$ can be deformed into the $(n-1)$-skeleton $B\pi^{(n-1)}$ of a CW-complex structure on $B\pi$ and it is called {\em essential} otherwise.
\end{defn}

Note that for an inessential $n$-manifold $M$ we have $\dim_{mc}\Wi M\le n-1$. Indeed, a lift $\Wi{u^M}:\Wi M\to E\pi^{(n-1)}$ of a classifying
map is a uniformly cobounded proper map to an $(n-1)$-complex.
Generally, if a classifying map $u^M:M\to B\pi$ can be deformed to the $k$-dimensional skeleton, then $\dim_{mc}\Wi M\le k$.

The following was proven in~\cite{BD}:
\begin{prop}[\cite{BD}, Lemma 3.5]\label{ref2}
For an inessential manifold $M$
with a CW complex structure a classifying map $u:M\to B\pi$ can be chosen such that
$$u(M^{(n-1)})\subset B\pi^{(n-2)}.$$
\end{prop}

\subsection{Macroscopically inessential manifolds}
The macroscopic dimension can be used for the description of the macroscopic inessentiality of a manifold.

\begin{thm}[\cite{Dr1}]\label{ref-mc}
Let $M$ be a closed oriented $n$-manifold  and let $\Wi u:\Wi M\to E\pi$ be a lift of $u^M:M\to B\pi$.
Then the inequality $\dim_{mc}\Wi M\le n-1$ is equivalent to the following condition:

(*) The map $\Wi u$ can be deformed by a bounded homotopy to $E\pi^{(n-1)}$ for any lift of a proper metric on $B\pi$.
\end{thm}
Note that we may assume that $B\pi$ is a locally finite complex and hence it admits a proper metric.
Also note that $\Wi u$ admits a bounded deformation to a map $f:\Wi M\to E\pi$ if and only if $f$ is in bounded distance to $\Wi u$, i.e. there is $C>0$ such that
$d(f(x),\Wi u(x))< C$ for all $x\in\Wi M$.

Thus, the inequality  $\dim_{mc}\Wi M^n\le n-1$ is a macroscopic analog of inessentiality. We call manifolds
$N$ with $\dim_{mc}N<\dim N$  {\em macroscopically inessential}.

There is an analog of Proposition~\ref{ref2}:
\begin{prop}[\cite{Dr1}, Lemma 5.3.]\label{ref3}
For an $n$-dimensional manifold $M$ with a fixed CW complex structure, a classifying map $u:M\to B\pi$ and with macroscopically
inessential universal covering $\Wi M$
any lift $\Wi u:\Wi M\to E\pi$ of $u$ admits a bounded deformation to a proper map $f:\Wi M\to E\pi^{(n-1)}$ with
$f(\Wi M^{(n-1)})\subset E\pi^{(n-2)}.$
\end{prop}

\subsection{Macroscopic dimension and QI-embeddings}

A map $f:X\longrightarrow Y$ between metric spaces is a quasi-isometric embedding (QI-embedding) if for all $x_1, x_2$ in X, there are $\lambda, c >0$, such that $$\frac{1}{\lambda}d_X(x_1,x_2)-c\leq d_Y(f(x_1),f(x_2))\leq\lambda d_X(x_1,x_2)+c.$$
\newline Note that any countable CW complex $B$ admits a metric $d$, called a \emph{weak metric} on $B$, such that the identity map $id: B\longrightarrow (B,d)$ is continuous. It means that $d$ does not necessarily define the CW topology on $B$ but its restriction to any compact set does.
\newline\indent The following theorem, proved in \cite{BD}, characterizes macroscopic dimension of universal coverings of finite CW complexes by continuous QI-embeddings.
\begin{thm}[\cite{BD}, Theorem 2.12]\label{Macroscopic dimension and QI-embeddings}
  Suppose that $E\pi$ is given a metric $d$ lifted from a weak metric on $B\pi$ where $\pi=\pi_1(X)$ for a finite CW complex X. Then the inequality $dim_{mc}\Wi X\leq n$ holds true for the universal covering $\Wi X$ of $X$ with the lifted metric if and only if there is a continuous QI-embedding $g:\Wi X\longrightarrow E\pi^{(n)}$.
\end{thm}

 It was shown in the proof of the above theorem that given a QI-embedding $g:\Wi X\longrightarrow E\pi^{(n)}$, then it is proper and uniformly cobounded. It follows from the above theorem and Proposition 2.4 that for a closed $n$-manifold $M$, a lift $\Wi u:\Wi M\longrightarrow E\pi$ of $u:M\longrightarrow B\pi$ admits a bounded deformation to a QI-embedding $f:\Wi M\longrightarrow E\pi^{(n-1)}$ with $f(\Wi M^{(n-1)})\subset E\pi^{(n-2)}$.

\subsection{Bordism}
We recall that the group of oriented relative bordisms    $\Omega_n(X,Y)$
of  the pair $(X,Y)$ consists of the equivalence classes of pairs $(M,f)$
where $M$ is an oriented $n$-manifold with boundary and $f:(M,\partial M)\to (X,Y)$ is a continuous map.  Two pairs $(M,f)$ and $(N,g)$ are equivalent if there is a pair $(W,F)$, $F:W\to X$  called a {\em bordism} where $W$ is an orientable $(n+1)$-manifold with boundary such that  $\partial W =
M\cup W'\cup N$, $W'\cap M=\partial M$, $W'\cap N=\partial N$, $F|_M=f$,  $F|_N=g$, and $F(W')\subset Y$.

In the special case when $X$ is a point, the manifold $W$ is called a bordism between $M$ and $N$.
The following proposition is proven in~\cite{BD}.

\begin{prop}\label{bordism} For any CW complex $K$
there is an isomorphism $$\Omega_n(K,K^{(n-2)})\cong H_n(K,K^{(n-2)}).$$
\end{prop}

In this paper we consider bordisms of open manifolds obtained from an infinite family of bordisms of compact manifolds with boundary
as follows. Let $M$ be an open $n$-manifold with a family of disjoint $n$-dimensional submanifolds with boundary $\{V_\gamma\}$.
Let $U_\gamma$ be a family of stationary on the boundary bordisms between $V_\gamma$ and $N_\gamma$.
Then the manifold $W=A\cup B$ with
$A=(M\setminus\cup_i\Int V_\gamma)\times[0,1]$ and $B=\coprod_\gamma U_\gamma$ with
$A\cap B=\coprod_\gamma\partial V_\gamma\times[0,1]$ is a {\em bordism
defined by the family} $\{U_\gamma\}$. Thus, $W$ is obtained from $M\times[0,1]$ by replacing
the cylinders $V_\gamma\times[0,1]$ by bordisms $U_\gamma$.
We call $W$ a bordism
between manifolds $M=M\times\{0\}$ and $N$ where  $\partial W=M\coprod N$.
\newline We consider the product metric on $M\times[0,1]$.
For each $\gamma$ one can define a metric on $U_\gamma$ that extends the metric from $\partial V_\gamma\times[0,1]$. This gives a metric on $W$
such that the inclusion $M\subset W$ is an isometric embedding.

\subsection{Poincare-Lefschtz Duality} The following theorem is classic (\cite{Ha}, p.254).
\begin{thm}[Poincare-Lefschetz duality]
Suppose $M$ is a compact orientable $n$-manifold whose boundary is decomposed as the union of two compact $(n-1)$-dimensional manifolds $A$ and $B$ with a common boundary $\partial A=\partial B=A\cap B$. Then there is an isomorphism $$D:H^k(M,A)\to H_{n-k}(M,B)$$ for all $k$.
\end{thm}

\

\subsection{A Note on Obstruction Theory}

\

Let $(X,Y)$ be a CW-pair and $g$ a map from $Y$ to $Z$. Suppose that $Z$ is path connected and $\pi_1(Z)$ acts trivially on $\pi_n(Z)$ for every $n$. The aim is to extend the map $g$ to a map from $X$ to $Z$ by proceeding inductively over the skeleton of $X$. Suppose $g$ has been extended to a map $\hat{g}$ on the $n$-skeleton $X^{(n)}$, $\hat{g}:X^{(n)}\rightarrow Z$, then the obstruction $[\sigma_{\hat{g}}]$ to extending $\hat{g}|_{X^{(n-1)}}$ to $X^{(n+1)}$ lies in the cohomology group $H^{n+1}(X,Y;\pi_n(Z))$. Obstruction theory says that $\hat{g}|_{X^{(n-1)}}:X^{(n-1)}\rightarrow Z$ extends to a map $X^{(n+1)}\rightarrow Z$ iff $[\sigma_{\hat{g}}]=0$ (See \cite{DK}, ch.7). Here, $[\sigma_{\hat{g}}]$ is generated by the cochain $\sigma_{\hat{g}}\in C^{n+1}(X,Y;\pi_n(Z))$ defined on the $(n+1)$-cells, generators of $C_{n+1}(X,Y)$, by $\sigma_{\hat{g}}(e^{(n+1)}_i)=[\hat{g}\circ\phi_i|_{S^n}:S^n\rightarrow Z]\in\pi_n(Z)$, where $\phi_i:D^{(n+1)}_i\rightarrow X$ is the characteristic map. It easily follows that $\sigma_{\hat{g}}=0$ if and only if $\hat{g}$ extends to a map $X^{(n+1)}\rightarrow Z$.

\

\

\section{Proof of Main theorem}

The condition of residual finiteness in the main theorem comes into the picture in the following lemma:

\begin{lemma}\label{immersion}
Let $M$ be a 4-manifold with a residually finite fundamental group. Then for any $\alpha\in\pi_2(M)$ there is a finite covering $p':M'\to M$ and an immersion $h':S^2\to M'$ such that
$p'_*([h'])=\alpha$ and the inclusion homomorphism of the fundamental groups
$\pi_1(h'(S^2))\to\pi_1(M')$ is trivial.
\end{lemma}
\begin{proof}
Let $h:S^2\to M$ be an immersion without triple points that realizes $\alpha$. Let $\tilde h:S^2\to\Wi M$ be its lift with respect to the universal covering $p:\Wi M\to M$.
Since $\tilde h$ is an immersion, there are finitely many pairs of points $(x^-_1,x^+_1),\dots,(x^-_n,x^+_n)$ in $\tilde h(S^2)$ with $p(x^-_i)=p(x^+_i)$ such that
$p$ restricted to $\tilde h(S^2)\setminus(\cup_{i=1}^n\{x^\pm_i\})$ is 1-to-1. There are $g_i\in\pi=\pi_1(M)$ with $g_i(x_i^-) = x_i^+$. Since $\pi$ is residually finite,
there is an epimorphism $\phi:\pi\to F$ to a finite group such that $\phi(g_i)\ne e$ for all $i$ and $\phi$ restricted on $\{g_1, ..., g_n\}$ is injective. Then the universal covering is factorized as $p=p'\circ q$ with $p':M'\to M$ a finite covering
and $q:\Wi M\to M'$ the projection onto orbit space of the action of $ker\phi$. Since $q$ is $\pi$-equivariant map, $q(x_i^-)\ne q(x_i^+)$ for all $i$. Therfore, $q$ restricted to $\tilde h(S^2)$ is a homeomorphism onto the image. We consider $h'=q\circ\tilde h$. The commutative diagram
$$
\begin{CD}
\pi_1(\tilde h(S^2)) @>>> \pi_1(\Wi M)\\
@Vq|_{\dots}VV @VqVV\\
\pi_1(h'(S^2)) @>>>\pi_1(M')\\
\end{CD}
$$
and the fact that $\pi_1(\Wi M)=0$ imply the required property of $h'$.
\end{proof}

\begin{prop}\label{ker}
Let $\phi:A\to\Z_2$ be a $\pi$-module homomorphism where $A$ is a finitely generated $\pi$-module. Then the kernel $\ker\phi$ is a finitely generated $\pi$ module.
\end{prop}
\begin{proof}

Let $A$ be generated as a $\pi$-module by $S=\{x_1, x_2, ..., x_n\}$ and let $F$ be the free $\pi$-module on the set $S$ and $A$ be the homomorphic image of $F$ . Consider the kernel of the $\pi$-module homomorphism $\phi\circ p:F\longrightarrow \Z_2$, where $p:F\longrightarrow A$. Then $Ker(\phi\circ p)$ is generated as a $\pi$-module by the set $S'=\{x_1^{i_1}, ..., x_n^{i_n}\}, i_j = 1,2$. Hence, $Ker(\phi\circ p)$ is finitely generated. Since $p|_{Ker(\phi\circ p)}:Ker(\phi\circ p)\longrightarrow Ker(\phi)$ is surjective, it follows that $Ker(\phi)$ is finitely generated as a $\pi$-module.

\end{proof}

Let  $\nu_M:M\to BSO$ denote
a classifying map for the stable normal bundle of  a  manifold $M$.

\begin{thm}\label{main}  For a closed smooth orientable totally non-spin $4$-manifold $M$ with residually finite fundamental group
the inequality $\dim_{mc}\Wi M\le 3$ for the universal covering $\Wi M$ implies the inequality $\dim_{mc}\Wi M\le 2$.
\end{thm}
\begin{proof} Let $\pi=\pi_1(M)$.
The idea of the proof is to perform surgery on a family $\{V_\gamma\}_{\gamma\in\Gamma}$ of compact submanifolds of $\widetilde M$  for some finite index subgroup $\Gamma\subset\pi$ to obtain a family of stationary on the boundary bordisms $\{(W_\gamma , q_\gamma)\}_{\gamma\in\Gamma}$ between $\{V_\gamma\}$ and $\{N_\gamma\}$ , where $q_\gamma(N_\gamma)\subset E\Gamma^{(2)}.$ Then use obstruction theory to extend the maps ${q_\gamma}$ to a map $q:W \longrightarrow E\Gamma^{(2)}$, where $W=A\cup B$ with $A=(\widetilde M\backslash \bigcup_\gamma\Int V_\gamma) \times [0,1]$ and $B=\coprod_\gamma W_\gamma$ with $A\cap B=\coprod_\gamma \partial V_\gamma \times [0,1]$.

\bigskip
The assumption that $\widetilde M$ is non-spin implies that the homomorphism $$(\nu_{\widetilde M})_*:\pi_2(\widetilde M)\to\pi_2(BSO)=\mathbb{Z}_2$$ is surjective.
Let $ h :S^2\longrightarrow\widetilde M$ be the immersion representing an element $[h]\in\pi_2(\Wi M)$ such that $(\nu_{\widetilde M})_*[ h]\neq 0$.
By Lemma~\ref{immersion} there is a finite index normal subgroup $\Gamma\subset\pi$ and corresponding finite covering $p':M'\to M$ such that
the homomorphism $\pi_1(q\circ h(S^2))\to\pi_1(M')=\Gamma$ is trivial where $q:\Wi M\to M'$ is the universal covering of $M'$. Let $V$ be a regular neighborhood of $qh(S^2)$. Thus $V$ is a 4-manifold with boundary which can be deformed to $qh(S^2)$.
Since the inclusion  homomorphism $\pi_1(V)\to\pi_1(M')$ is trivial, there is a section $s:V\to\Wi M$ of the universal covering $q:\Wi M\to M'$.
Let $V_\gamma=\gamma(s(V)),\gamma\in\Gamma$, denote the $\gamma$-translate of $s(V)$. Note that the family $\{V_\gamma\}$ is disjoint.

We may assume that $M'$ has a CW structure with one $4$-dimensional cell. Since $\dim_{mc}\widetilde M\le 3$, then by Proposition~\ref{ref3}, there is a bounded deformation of a lift $\widetilde u:\widetilde M\longrightarrow E\Gamma$ of a classifying map $u:M\longrightarrow B\Gamma$ to a map $f:\widetilde M\longrightarrow E\Gamma^{(3)}$ such that $f(\widetilde M\backslash\coprod_{\gamma\in\Gamma}D_\gamma)\subset E\Gamma^{(2)}$, where $\{D_\gamma\}_{\gamma\in\Gamma}$ are the translates of a section of $q$ over a fixed closed 4-ball $D$ in the 4-dimensional cell of $M'$. We may assume that $D\subset V$ and $D_\gamma\subset V_\gamma$. This can be done while keeping the images $f(D_\gamma)$ uniformly bounded. Clearly, $f$ is a QI-embedding.
\newline \indent Note that the restriction of $f$ to $(D_{\gamma},\partial D_{\gamma})$ defines a zero element in $H_4(E\Gamma,E\Gamma^{(2)})$. Since the family of sets $f(D_\gamma)$ is uniformly bounded,  there is $r>0$ such that $f(D_{\gamma})\subset B_r(f(c_{\gamma}))$ where $c_{\gamma}\in D_{\gamma}$ and
$f|_{D_{\gamma}}$ defines a zero element in $H_4(B_r(f(c_{\gamma})), B_r(f(c_{\gamma}))\cap E\Gamma^{(2)}).$

By Proposition~\ref{bordism}, there is a relative bordism
$(W_{\gamma},q_{\gamma})$ of $(D_{\gamma},\partial D_{\gamma})$ to $(N'_{\gamma},S'_{\gamma})$ with $q_{\gamma}(\partial W_{\gamma}\setminus D_{\gamma})\subset E\Gamma^{(2)}$ and $q_\gamma(N'_\gamma)\subset Br(f(c_\gamma))\cap E\Gamma^{(2)}$.
We may assume that the bordism $W'_{\gamma}\subset\partial W_{\gamma}$ of the boundaries $\partial D_{\gamma}\cong S^3$ and $S'_{\gamma}$ is stationary; $W'_{\gamma}\cong\partial D_{\gamma}\times[0,1]$
and $q(x,t)=q(x)=f(x)$ for all $x\in\partial D$ and all $t\in[0,1]$. Applying 1-surgery we may assume that $W_{\gamma}$ is simply connected.

For each $\gamma\in\Gamma$, we can enlarge the bordism $W_\gamma$ by the trivial bordism $(V_\gamma\backslash\Int D_\gamma)\times [0,1]$ to a bordism $\overline{W}_\gamma$ of manifolds with boundaries between $V_\gamma$ and $N_\gamma$. The maps $q_\gamma$ extend by means of $f$ to maps $\overline{q}_\gamma :\overline{W}_\gamma\longrightarrow E\Gamma$ with $\overline{q}_\gamma(V_\gamma\backslash\Int D_\gamma)\times[0,1]\subset E\Gamma^{(2)}$. Thus $(\overline{W}_\gamma,\overline{q}_\gamma)$ is a bordism between $V_\gamma$ and $N_\gamma$ that is stationary on the boundary and $\overline{q}_\gamma(N_\gamma)\subset E\Gamma^{(2)}$. Note that the inclusion $V_\gamma\to\overline{W_\gamma}$ induces an isomorphism of the fundamental groups.

Since the images of $N_\gamma$ are uniformly bounded, there are finite isometric subcomplexes $Z_\gamma, \gamma\in\Gamma$, of $E\Gamma^{(2)}$ and $b>0$ sufficiently large such that $\overline{q}_\gamma(N_\gamma)\subset Z_\gamma$ and $diam( Z_\gamma)<b$. Now consider a finite set of loops $g^\gamma_i:S^1\to Z_\gamma$ that generates
 $\pi_1(Z_\gamma)$,
 and define $$X_\gamma=Z_\gamma\cup_{\cup g^\gamma_i}(\coprod_i D^2_i)$$ be the space obtained by attaching disks to $Z_\gamma$ along the maps
$g^\gamma_i$. Then $X_\gamma, \gamma\in\Gamma$, are simply connected and isometric. Note that since $\pi_1(E\Gamma^{(2)})=0$, the inclusion map $j_\gamma:Z_\gamma\longrightarrow E\Gamma^{(2)}$ can be extended to a map $\overline{j_\gamma}:X_\gamma\longrightarrow E\Gamma^{(2)}$. Moreover, we can assume that $diam(\overline{j_\gamma}(X_\gamma))<b$.

 Consider the induced homomorphism $(\nu_{\overline{W}_\gamma })_*:\pi_2(\overline{W}_\gamma)\to\pi_2(BSO)=\mathbb{Z}_2$. This homomorphism is surjective and every 2-sphere $S$ that generates an element of the kernel has a trivial stable normal bundle. Since $\pi_1(\overline{W}_\gamma)\cong\pi_1(V_\gamma)$ is a finitely generated free group, $\pi_2(\overline{W}_\gamma)$ is a finitely generated $\pi_1(\overline{W}_\gamma)$-module (see~\cite{Ra}). It follows from Proposition~\ref{ker} that the kernel of $(\nu_{\overline{W}_\gamma })_*$ is finitely generated. Hence we can perform 2-surgery on $\overline{W}_\gamma$ to obtain a bordism $(\hat{W}_\gamma,\hat{q}_\gamma)$ between $V_\gamma$ and $N_\gamma$ and a map $\nu_{\hat{W}_\gamma}:\hat{W}_\gamma\longrightarrow BSO$ with $\hat{q}_\gamma=\overline{q}_\gamma=f$ on $\partial\hat{W}_\gamma=\partial\overline{W}_\gamma$ and with a uniform bound for the distance between images of $\hat{q}_\gamma$ and $\overline{q}_\gamma$. Hence, we can assume that $\hat{q}_\gamma(N_\gamma)\subset Z_\gamma$.

 Let $i_\gamma:V_\gamma\longrightarrow\hat{W}_\gamma$ denote the inclusion map. Then $(\nu_{\hat{W}_\gamma})_*\circ i_{\gamma*}=(\nu_{V_\gamma})_*$. It follows that $i_{\gamma*}:\pi_2(V_\gamma)\longrightarrow\pi_2(\hat{W}_\gamma)$ is surjective since  $(\nu_{V_\gamma})_*$ is surjective and $(\nu_{\hat{W}_\gamma})_*$ is an isomorphism. Thus, by the exact sequence of pairs $(\hat{W}_\gamma,V_\gamma)$, we get $\pi_2(\hat{W}_\gamma,V_\gamma)=0$. By the relative Hurewicz theorem, we get $H_1(\hat{W}_\gamma,V_\gamma)=H_2(\hat{W}_\gamma,V_\gamma)=0$.

Since $\hat{W}_\gamma$ is a smooth 5-manifold with corners, then it can be triangulated and hence it is a CW complex. Therefore, we can use obstruction theory on $(\hat{W}_\gamma, N_\gamma\bigcup\partial V_\gamma\times[0,1])$ to extend the map $i_\gamma\circ\hat{q}_\gamma|_{N_\gamma\bigcup\partial V_\gamma\times[0,1]}:N_\gamma\bigcup\partial V_\gamma\times[0,1]\longrightarrow X_\gamma$ to a map from $\hat{W}_\gamma$ to $X_\gamma$, where $i_\gamma:Z_\gamma\longrightarrow X_\gamma$ is the inclusion map.
Since $X_\gamma$ is simply connected, we can extend the map to the 2-skeleton $\hat W_\gamma^{(2)}$. The primary obstruction for this extension problem lives in $H^3(\hat W_\gamma,N_\gamma\bigcup\partial V_\gamma\times[0,1];\pi_2(X_\gamma))=H^3(\hat W_\gamma,N;\pi_2(X_\gamma))$. By the Poincare-Lefschetz Duality, $H^3(\hat W_\gamma,N;\pi_2(X_\gamma))=H_2(\hat W_\gamma,V_\gamma;\pi_2(X_\gamma))$, $H^3(\hat W_\gamma,N_\gamma)=H_2(\hat W_\gamma, V_\gamma)=0$ and $H^4(\hat W_\gamma,N_\gamma)=H_1(\hat W_\gamma, V_\gamma)=0$. By the Universal Coefficient Formula,
$$H^3(\hat W_\gamma,N_\gamma;\pi_2(X_\gamma))=H^3(\hat W_\gamma,N_\gamma)\otimes\pi_2(X_\gamma)\oplus Tor(H^4(\hat W_\gamma,N_\gamma),\pi_2(X_\gamma))=0.$$
Thus, there is an extension of the map $i_\gamma\circ\hat q_\gamma|_{N_\gamma\bigcup\partial V_\gamma\times[0,1]}$ to $\hat W_\gamma^{(3)}$. The second obstruction lives in the group $H^4(\hat W_\gamma,N_\gamma;\pi_3(X_\gamma)).$
By the Universal Coefficient Formula, $$H^4(\hat W_\gamma,N_\gamma;\pi_3(X_\gamma))=H^4(\hat W_\gamma,N_\gamma)\otimes\pi_3(X_\gamma))=0.$$
And the third obstruction lives in $H^5(\hat W_\gamma,N_\gamma;\pi_4(X_\gamma))= H_0(\hat W_\gamma,V_\gamma;\pi_4(X_\gamma))=0$.
Thus, there is an extension of $i_\gamma\circ\hat{q}_\gamma|_{N_\gamma\bigcup\partial V_\gamma\times[0,1]}:N_\gamma\bigcup\partial V_\gamma\times[0,1]\longrightarrow X_\gamma$ to a map $\tilde q_\gamma:\hat W_\gamma\to  X_\gamma$.

 Let $\hat{W}=A\cup B$, where $A=(\widetilde M\backslash\cup_\gamma\Int V_\gamma)\times[0,1]$ and $B=\coprod_\gamma\hat{W}_\gamma$ with $A\cap B=\coprod_\gamma\partial V_\gamma\times[0,1]$ and let $i:\widetilde M\longrightarrow \hat{W}$ denote the inclusion map. We may choose a metric on $\hat{W}$ such that $i$ is an isometric embedding.
Then the maps $\cup_\gamma\overline{j_\gamma}\circ \tilde q_\gamma:B\to E\Gamma^{(2)}$ naturally extend by means of $f$ to a QI-embedding $\hat{q}:\hat{W}\longrightarrow E\Gamma^{(2)}$.
The restriction $g=\hat q|_{\Wi M}:\Wi M\to  E\Gamma^{(2)}$ is in finite distance to $\tilde u:\Wi M\to E\Gamma$. Thus, by Theorem 2.5, $\dim_{mc}\Wi M\le 2$.
 \end{proof}
\begin{cor}
Gromov's Conjecture holds true for 4-manifolds with abelian fundamental groups.
\end{cor}
\begin{proof}
The case where $M$ is spin was proved in~\cite{BD2}, and the case where the universal cover $\Wi M$ is spin was proved in~\cite{Dr1}. Since abelian groups are residually finite, then in view of Theorem~\ref{main}, to complete the proof in the case where $\Wi M$ is non-spin, it suffices to show that a positive scalar curvature 4-manifold with abelian fundamental group is inessential.
Since Gromov's conjecture for a finite index subgroup implies the conjecture for the group~\cite{BD2}, we may assume that $\pi_1(M)$ is free abelian. If $rank(\pi_1(M))\ne 4$, then the manifold $M$ is inessential by definition. If $rank(\pi_1(M))=4$ and $M$ admits an essential map onto the 4-torus $T^4$, then $M$ does not admit a metric with positive scalar curvature~\cite{Sch},~\cite{SY}.
\end{proof}

\bigskip
\bigskip

\noindent\textbf{Question:} Let $\nu:M^4\longrightarrow BSO$ be a classifying map of the stable normal bundle. Does there exist an immersion $h:S^2\longrightarrow M^4$ such that $\nu_\ast[h_\ast(S^2)]\neq0$ in $\pi_2(BSO)$ and the inclusion $h(S^2)\longrightarrow M^4$ induces the trivial homomorphism on the fundamental groups?
\newline
\newline
\textbf{Remark:} In case of a positive answer to the above question, the residually finite condition on the fundamental group in the statement of the main theorem would drop.

\end{document}